\documentclass[12pt]{amsart}
\usepackage{epsfig}
\usepackage{amsmath}
\usepackage{amsthm, amssymb, amscd}
\usepackage{hyperref}
\DeclareSymbolFont{AMSb}{U}{msb}{m}{n}
\DeclareMathSymbol{\N}{\mathbin}{AMSb}{"4E}
\DeclareMathSymbol{\Z}{\mathbin}{AMSb}{"5A}
\DeclareMathSymbol{\R}{\mathbin}{AMSb}{"52}
\DeclareMathSymbol{\Q}{\mathbin}{AMSb}{"51}
\DeclareMathSymbol{\I}{\mathbin}{AMSb}{"49}
\DeclareMathSymbol{\C}{\mathbin}{AMSb}{"43}

\author[B.\ Davis, H.N.\ Howards, J.\ Newman, J.\ Parsley]{Bob Davis, Hugh Howards, Jonathan Newman, Jason Parsley}
\title{An infinite family of convex Brunnian links in $\R^n$}


\newtheorem{theorem}{Theorem}[section]

\newtheorem{lemma}[theorem]{Lemma}

\newtheorem*{mainthm}{Theorem~\ref{thm:cbl}}

\theoremstyle{definition}

\newtheorem{conjecture}[theorem]{Conjecture}

\newtheorem{q}[theorem]{Question}

\newcommand{\bi}{\begin{itemize}}
\newcommand{\ei}{\end{itemize}}
\newcommand{\be}{\begin{enumerate}}
\newcommand{\ee}{\end{enumerate}}

\renewcommand{\q}{\tilde{q}}
\def\minus{\hbox{--}}

\begin{document}
\begin{abstract}
This paper proves that convex Brunnian links exist for every dimension $n \geq 3$ by constructing explicit examples.  These examples are three-component links which are higher-dimensional generalizations of the Borromean rings.
\end{abstract}
\maketitle

\begin{figure}[h]
	\centering
	\includegraphics[width=0.35\textwidth]{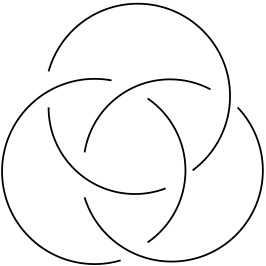}
	\caption{The Borromean Rings}
	\label{fig:bor}
\end{figure}

\section{Introduction}
The link depicted in Figure~\ref{fig:bor} is known as the Borromean rings and appears to consist of three round circles.  This, however, was proven to be an optical illusion by Mike Freedman and Richard Skora in \cite{MR873456}, who showed that at least one component must be noncircular.  A different proof of this result was given by Bernt Lindstr{\"o}m and Hans-Olov Zetterstr{\"o}m in \cite{MR1103189}; it seems they were not aware of the earlier result.  

Brunnian links were introduced over a hundred years ago by Hermann Brunn in his 1892 paper ``Uber Verkettung'' ({\it ``On Linking''}) \cite{brunn}. They have been generalized both in $\R^3$ as well as in higher dimensions.  Debrunner \cite{MR0137106} and Penney \cite{MR0238302} each looked at generalizations of Brunnian links in $\R^3$ in 1961 and 1969 respectively.  Takaaki Yanagawa was the first to look at higher-dimensional Brunnian links such as the ones we study in this paper back in 1964 when he constructed 2-spheres in $\R^4$ that formed Brunnian links \cite{MR0173256}.

Instead of linked circles in $\R^3$, one can consider linked spheres of various dimensions in $\R^n$.  In \cite{MR2384264} it is shown that no Brunnian link in $\R^n$ can ever be built out of round spheres for any $n \in \Z^+$.    

However, the Borromean rings can be built out of two circles and one ellipse that is arbitrarily close to a circle by using the equations 
\begin{align*}
K_1 & =  \{x_{1}^2 + x_{2}^2 = r_1^2, \, x_{3}=0 \} \\
K_2 & = \{ x_2^2+x_3^2  = r_2^2, \, x_{1}= 0 \} \\
K_3 & = \{ \frac{x_{1}^2}{(r_3)^2}+ \frac{x_{3}^2}{(r_4)^2}= 1,  \, x_{3}=0 \},
\end{align*} 
where $r_3 < r_1 < r_2 < r_4$, all arbitrarily close to each other.

In \cite{MR2287436} it is proven that although there are an infinite number of Brunnian links of 3 components in $\R^3$ (or any number of components $\geq 3$), the Borromean rings are the only Brunnian link in dimension three of either 3 or 4 components that can be built out of convex components.  The result was extended to 5 components in \cite{davis}.

The question of whether any Brunnian link  in $\R^n$ ($n \geq 3$) other than the Borromean rings can be built out of convex components was asked in \cite{MR2384264, davis}.
The main result of this  answers that question in the affirmative: 

\begin{mainthm}
Consider the infinite family $L_{i,n}$ of three-component links given explicitly by \eqref{firstlink}; each consists of a round $(n\minus 2)$-sphere, a round $(n\minus i \minus 1)$-sphere and an $i$-dimensional ellipsoid sitting in $\R^n$ (for $1 \leq i \leq n \minus 2$ and $n \geq 3$).  Each $L_{i,n}$ is a convex Brunnian link.
\end{mainthm}
As we will see, each of these links is a natural generalization of the Borromean rings.  Moreover, each can be constructed so the ellipsoid is arbitrarily close to being round.    

\smallskip

This  is organized as follows:  the next section covers the background and relevant definitions for higher-dimensional linking.  The main portion of this  is section~\ref{sec:main}, which proves Theorem~\ref{thm:cbl}.  

In section~\ref{sec:proof2}, we provide a second proof of Theorem~\ref{thm:cbl} for the special case $i=1$, in which we explicitly realize the first component as an ellipse (an $S^1$) and the other two as round $(n\minus 2)$-spheres.  This second proof uses the fundamental group as its main tool but does not extend to the general case.  Section~\ref{sec:open} concludes this  with questions and conjectures about other convex Brunnian links:  do they exist?  Do the Borromean rings generalize to three $(n\minus 2)$-spheres sitting in $\R^n$?

\section{Standard definitions}

Recall that a {\em knot} is a subset of $\R^3$ or $S^3$ that is homeomorphic
to a circle (also called a 1-sphere or $S^1$).  If the knot bounds an embedded disk
it is called an {\em unknot}; otherwise it is knotted.

A {\em link} $L$ is a collection of disjoint knots.  A link $L$ is an {\em unlink}
of $n$ components if it consists of $n$ unknots and if the components
simultaneously bound disjoint embedded disks.

A {\em Brunnian link} $L$ is a link of $n \geq 3$ components that is not an unlink, but every proper sublink of $L$ is an unlink.  The {\em Borromean rings} form the most famous example of a Brunnian link (Figure~\ref{fig:bor}).  Note that eliminating any one of the components yields an unlink.

A subset $K$ of $\R^n$ is a {\em knot in $\R^n$} if $K$ is
homeomorphic to $S^k$ for some $k$. By a {\em link in $\R^n$} is meant a subset $L$ of
$\R^n$ that is homeomorphic to a disjoint union
of finitely many knots (possibly of different
dimensions).  Knot theory is usually restricted to the case where $n=3$; each knot is homeomorphic to $S^1$.

A link $L = F_1 \cup F_2 \cup \dots F_m$ in  $\R^n$  is {\em an
unlink in $\R^n$} if for each $i$ the knot $F_i$ bounds a ball $B_i$ (of appropriate dimension) such that
$B_i \cap F_j = \emptyset$ ($i \neq j$).
If $n =3$ and we restrict to circles, we obtain the traditional definition of an unlink,
where each component bounds a disk disjoint from the other components.
If a link $L$ (in $\R^n$) of  $m (\geq 3)$ components
is not an unlink (in $\R^n$), yet every proper sublink is an unlink (in $\R^n$), we call $L$ a {\em Brunnian link in $\R^n$}.

We say a link $L$ is a {\em split link in $\R^n$} if there is an $(n-1)$-sphere that is disjoint from the link and separates $\R^n$ into two components, each containing at least one component of the link. (The $(n-1)$-sphere need not be round.)

A knot $S^k$ is said to be {\em convex} if it bounds a ball $B^{k+1}$ which is convex.

A {\em generalized Hopf link} is any link of two components, one an $S^j$ and the other an $S^{n-(j+1)}$ in $\R^n$ or $S^n$, each of which bounds a ball that intersects the other sphere transversally in exactly one point.  Examples include any link isotopic to the following link in $\R^n$:
 $F_1 = \{(x_1, x_2, \dots x_n): x_{i+1}^2+x_{i+2}^2 + \dots + x_n^2 = 1, x_1=x_2= \dots = x_i = 0 \}$
 and $F_2 = \{(x_1, x_2, \dots x_n): x_1^2 + x_2^2 + \dots + x_{i}^2 + (x_{n}-1)^2 = 1, x_{i+1}=x_{i+2}= \dots = x_{n-1}=0 \}$, where $1 \leq i \leq n-2$.

\section{An infinite family of convex Brunnian links}
\label{sec:main}

In this section, we present our main result, the existence of an infinite family of three-component convex Brunnian links in $\R^n$.  Define the family $L_{i,n}$ as
\begin{equation} 
\label{firstlink}
L_{i,n} = K_1 \cup K_2 \cup K_3 
\end{equation}
\begin{align*}
K_1 &= \{(x_1, x_2, \dots x_n): x_1^2 + x_2^2 + \dots + x_{n-1}^2 = 4, \; x_n=0 \} \\
K_2 &= \{(x_1, x_2, \dots x_n): x_{i+1}^2+x_{i+2}^2 + \dots + x_n^2 = 9, \; x_1=x_2= \dots = x_i = 0 \} \\
K_3 &= \{(x_1, x_2, \dots x_n): x_1^2 + x_2^2 + \dots + x_{i}^2 + \tfrac{x_{n}^2}{16} = 1, x_{i+1}= \dots = x_{n-1}=0 \},
\end{align*}
where $n \geq 3$ and $1 \leq i \leq n-2$.  

\begin{theorem} 
Each $L_{i,n}$ is a convex Brunnian link.
\label{thm:cbl}
\end{theorem}

We know that $L_{1,3}$ consists of two circles and an ellipse forming the Borromean rings, a Brunnian link.  In the general case it is clear that the components are convex.  The following lemma shows that all sublinks of $L_{i,n}$ are unlinks.  We must show that no $L_{i,n}$ is an unlink, which we accomplish via 
Lemmas~\ref{lemma:dr}-\ref{lemma:disjt}.  

Henceforth, we adopt the convention for specifying knots and balls that all omitted coordinates are set equal to be zero.  

\begin{lemma}
Every proper sublink of $L_{i,n}$ is an unlink.
\end{lemma}

\begin{proof}
The proper sublinks of $L_{i,n}$ are pairs of unknots.  For each pair, we observe that one knot bounds a ball disjoint from the other, and thus the pair forms a split link and must be an unlink.  

Explicitly, the round $(n \minus 1)$-ball $B_1 = \{x_{1}^2 + x_{2}^2 + \dots + x_{n-1}^2 \leq 4 \} \subset \R^n$ bounded by $K_1$ lies in the complement of the $(n \minus i \minus 1)$-sphere $K_2$.  Similarly, $K_2$ bounds a (round) ball $B_2$ disjoint from the $i$-dimensional ellipsoid $K_3$, and $K_3$ bounds an ellipsoidal ball $B_3$ disjoint from $K_1$. 
\end{proof}

 The next lemma, stated without proof, relays a standard fact about spheres.  

\begin{lemma}
\label{lemma:dr} 
Let $S^n \subset \R^{n+1}$ be a round sphere centered at the origin, and let $V$ be a linear subspace of $\R^{n+1}$.  Consider the great spheres $X_1=V \cap S^n$  and $X_2=V^\perp \cap S^n$.  Then, $S^n -X_1$ deformation retracts onto $X_2$.
\end{lemma}

In particular, if we delete an unknotted $S^{k-1} \subset S^n$ from $S^n$, the result deformation retracts onto an $S^{n-k}$.

Next, we note that $K_3 - (B_1 \cap K_3)$ is not connected, since the intersection $(B_1 \cap K_3)$ is exactly the equator $x_n = 0$ of the $i$-dimensional ellipsoid $K_3$.  The two components of $K_3 - (B_1 \cap K_3)$ are the upper and lower open halves of $K_3$.  We use the upper half for our next definition.  

Let $K_3'$ be the $i$-dimensional subset of $K_3 \cup B_1$ formed by taking the open upper half of the ellipsoid $K_3$ and closing it at the bottom with the disk $B_1 \cap B_3$; in coordinates,
\begin{align*}
K_3' = &  \left\{(x_1, x_2, \dots x_n): x_1^2 + x_2^2 + \dots + x_{i}^2 + \tfrac{x_{n}^2}{16} = 1, x_n \geq 0 \right\} \\  
& \; \cup \left\{(x_1, x_2, \dots x_n): x_1^2 + x_2^2 + \dots + x_{i}^2 \leq 1, x_n=0 \right\} .
\end{align*}

\begin{lemma}
$K_3' \cup K_2$ is a generalized Hopf link and is not a split link. 
\label{lemma:hopf}
\end{lemma}

\begin{proof}
Figure~\ref{fig:db3} depicts this link for $n=3$; notice that $K_3'$ orthogonally intersects $B_2$ at the origin, so this a Hopf link.  In arbitrary dimensions, the same phenomenon occurs:   $K_3'$ orthogonally intersects $B_2$ in only one point, the origin, and we have a generalized Hopf link.  

(By letting $B_3'$ be the portion of $B_3$ with $x_n \geq 0$, i.e., the portion bounded by $K_3'$, we also see that $K_2$ orthogonally intersects $B_3'$ in only one point,  $q_+=(0,0, \dots ,0,0,3)$.)

Now we show the link is not split.  Although $K'_3 \cup K_2$ lies in $\R^n$, our argument is more easily made in $S^n$.  If we include $\R^n$ inside $\R^{n+1}$ by fixing $x_{n+1}=0$, then we may use stereographic projection $p$ to lift $\R^n$ to the $n$-sphere $S^n \subset \R^{n+1}$ of radius 3, centered at the origin.  
(Usually the unit sphere is used for stereographic projection, but in this case it is more convenient to use the radius of $K_2$.)  We will show the lift of link $K'_3 \cup K_2$ is not split. 

\begin{figure}
	\centering
	\includegraphics[width=0.6\textwidth]{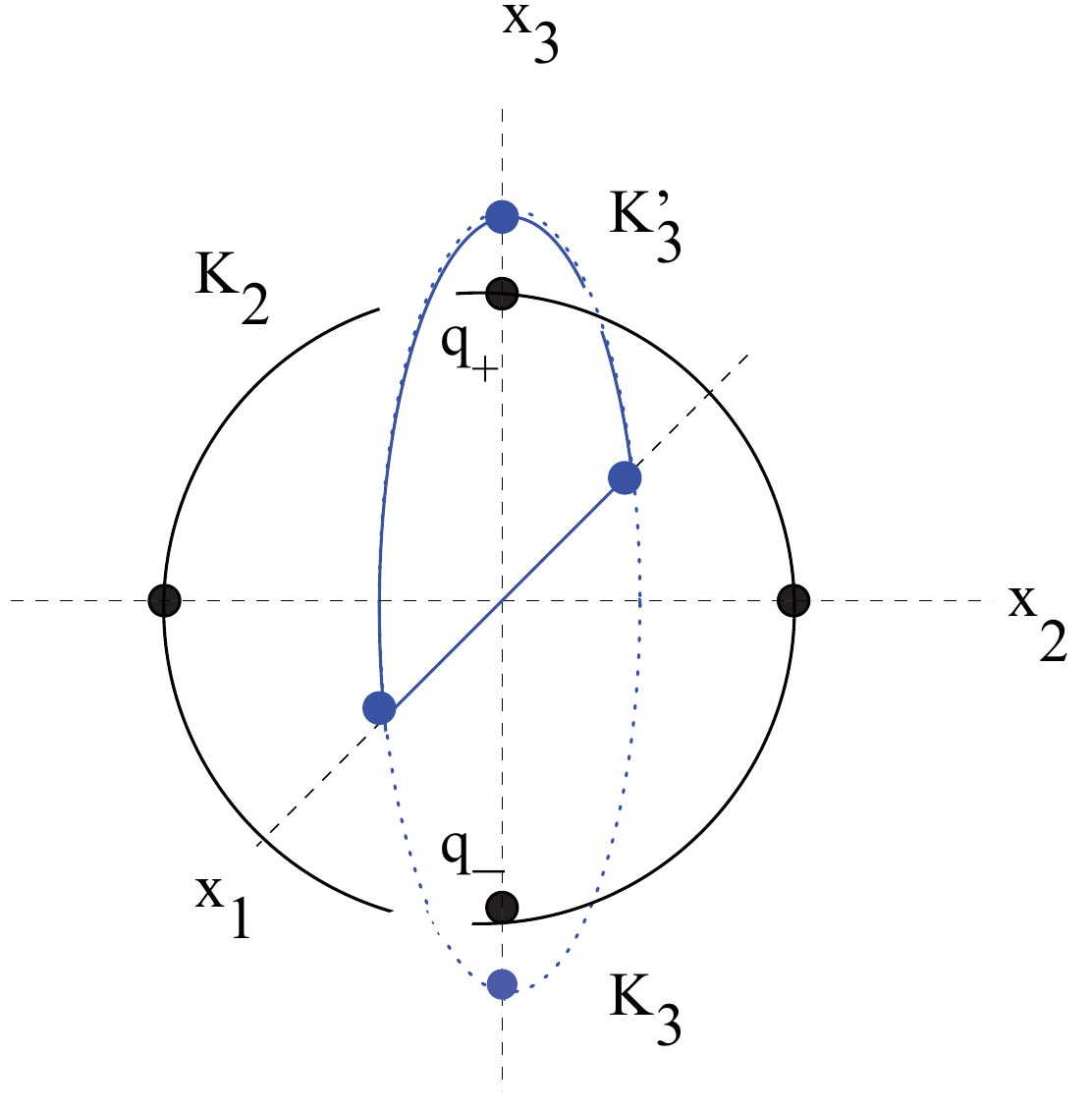}
	\caption{Here is how the argument looks for $L_{1,3} \subset \R^3$.  Ball $B_2$ has dimension $n-i=2$; it forms the round disk bounded by $K_2$ in the $x_2x_3$-plane; the $(i+1=2)$-ball $B_3'$ is the upper half-disk bounded by $K_3'$ in the $x_1x_3$-plane.  Let $p$ be the map via stereographic projection from  $\R^3$ to $S^3 \subset \R^4$.  Then we see that $p(K_2)=G_2$, $p(K'_3)=G'_3$, $p^{-1}(G_3)$ is the 
$x_1$-axis, and $p^{-1}(\Sigma)$ is the $x_1x_3$-plane.  Note that $q_+$ is `inside' $K_3'$ in the $x_1x_3$-plane, whereas $q_-$ is `outside'.}
	\label{fig:db3}
\end{figure}

Because $K_2$ has radius 3 and is centered at the origin, it is fixed by the lift $p$.  We consider these subsets of $S^n$:
\begin{itemize}
\item $G_2$, the great $(n\minus i \minus 1)$-sphere $p(K_2)$
\item $G_3$, the great $i$-sphere $\left\{x_{1}^2+x_{2}^2 + \dots + x_i^2 + x_{n+1}^2 = 9 \right\}$ in $S^n$ complementary to $G_2$
\item $G_3'$, the lift $p(K_3')$, and
\item $\Sigma$, the great $(i+1)$-sphere $\left\{ x_{1}^2+x_{2}^2 + \dots + x_i^2+x_n^2+ x_{n+1}^2 = 9 \right\}$ 
\end{itemize}

We note that both $G_3$ and $G_3'$ are unknotted $i$-spheres contained in $\Sigma \subset S^n$; the former inclusion is immediate while the latter follows since $K_3'$ lies in the $(i+1)$-dimensional subspace $p^{-1}(\Sigma)$.  We observe that $G_2 \cap \Sigma$ is a great 0-sphere consisting of two points $\q_\pm = (0,0, \dots, 0, \pm 3,0)$ which are disjoint from  $G_3 \cup G_3'$.  Note that $\q_\pm = p(q_\pm)$.

{\it Claim:} $G_3'$ is isotopic to $G_3$ in $\Sigma - G_2$ (which implies they are isotopic in $S^n - G_2$).

To prove the claim, it suffices to show both $G_3$ and $G_3'$ separate $\q_+$ from $\q_-$ in $\Sigma$, i.e, they can both be oriented to contain $\q_+$ on the inside and $\q_-$ on the outside.  This follows immediately for $G_3$ since it is a great $i$-sphere disjoint from antipodal points $\q_\pm$ in $\Sigma$.

Now we show $G_3'$ separates $\q_\pm$.  Recall that $K_3'$ bounds $B_3'$, which contains $q_+$ but not $q_-$.  Also note that $B'_3 \subset p^{-1}(\Sigma)$.  This property is preserved under homeomorphism $p$ since $p$ is a bijection and $B'_3 \subset p^{-1}(\Sigma)$ and $p(B'_3) \subset \Sigma$. The $i$-sphere $G_3'$ bounds $p(B_3')$, which contains $\q_+$ but not $\q_-$.  Thus $G_3'$ separates $\q_\pm$ and the claim holds.  Since $G_3$ and $G_3'$ are isotopic in $\Sigma - G_2$, they are also isotopic in the larger space $S^n - G_2$ which contains $\Sigma - G_2$.

%
%
%
%

We have now shown that $K_2 \cup K_3'$ lifts to a link isotopic to $G_2 \cup G_3$.
To prove the former is not split, we show the latter is not split by showing that $G_3$ does not bound an $(i+1)$-ball in the complement of $G_2$.   
Lemma~\ref{lemma:dr} assures us that $S^n - G_2$ deformation retracts onto $G_3$; this implies that $\pi_i(S^n - G_2) =\pi_i(G_3)=\Z$.  (It is well known that the $i$th homotopy group of an $i$-sphere is $\Z$; see, for example, \cite{MR1867354}
).  Since $G_3$ is fixed by the deformation retract and generates $\pi_i$ after the deformation retract, it must also generate $\pi_i$ before.  As a nontrivial element of $\pi_i(S^n - G_2)$, $G_3$ cannot bound a $(i+1)$-ball in $S^n - G_2$.  Therefore, $G_2 \cup G_3 \subset S^n$ is not a split link, and neither is $K_2 \cap K_3' \subset \R^n$.
\end{proof}

\begin{lemma}  \label{lemma:disjt}
If $K_2$ bounds an embedded $(n-i)$-ball $D_2$ which does not intersect $K_1 \cup K_3$, then $K_2$ bounds an immersed $(n-i)$-ball $D_2'$ that does not intersect $K_1 \cup K_3 \cup B_1$.
\end{lemma}

\begin{proof}
Let $D_2$ be an embedded $(n-i)$-ball bounded by $K_2$ that is disjoint from $K_1 \cup K_3$ and intersects $B_1$ transversally.  We want to show there is an immersed $(n-i)$-ball that is disjoint from $B_1$.  If $D_2$ is disjoint from $B_1$, we are done.  
If not, note that $B_1$ has dimension $n-1$ and thus has codimension 1 in $\R^n$.
Since $\partial B_1 \cap D_2 = \emptyset$ and $\partial D_2 \cap B_1 = \emptyset$, the set $B_1 \cap D_2$ must be a collection of disjoint closed manifolds $\{F_1, F_2, \dots  \}$ of dimension $n-i-1$.  

We note that since $D_2$ is a ball, each of the $F_i$ are separating in $D_2$.
We may define the {\em outside} of $F_i$ to be the component of $D_2 - F_i$ which contains $K_2$ and the {\em inside} to be the other component. 

Since $D_2$ is compact, we may assume that it has a finite number of critical points with respect to $x_n$. Since $B_1$ lies in the plane $x_n = 0$, we can conclude that the intersection $B_1 \cap D_2$ has a finite number of components.  
Because there are a finite number of intersections, we may take an innermost component $F_j$ (one which has no other $F_i$ inside of it in $D_2$).  Let $U$ be the component of $D_2 - F_j$ inside of $F_j$.

Let $f(x_1, x_2, \dots x_{n-1}, x_n)=(x_1, x_2, \dots x_{n-1}, |x_n|)$.
Note that $p$ is a point in $K_i$ if and only if $f(p)$ is a point in $K_i$ since each of the knots is symmetric with respect to $x_n$.
This implies that if $U \subset D_2$ then since $U \cap (K_1 \cup K_3) = \emptyset $, we know $f(U) \cap (K_1 \cup K_3) = \emptyset$.

Then we replace $D_2$ by $(D_2 - U) \cup f(U)$.  See Figures~\ref{fig:int1}-\ref{fig:int2}.  The new ball is not in general position with respect to $B_1$.  We may take a small deformation of the new (possibly no longer embedded)
ball to decrease the number of intersections with $B_1$.  We repeat, reducing the number of intersections each time, until we have a new ball $D_2'$ with boundary
$K_2$ that is disjoint from $B_1$ and $K_1 \cup K_3$.
\end{proof}

\begin{figure}
	\centering
	\includegraphics[width=0.5\textwidth]{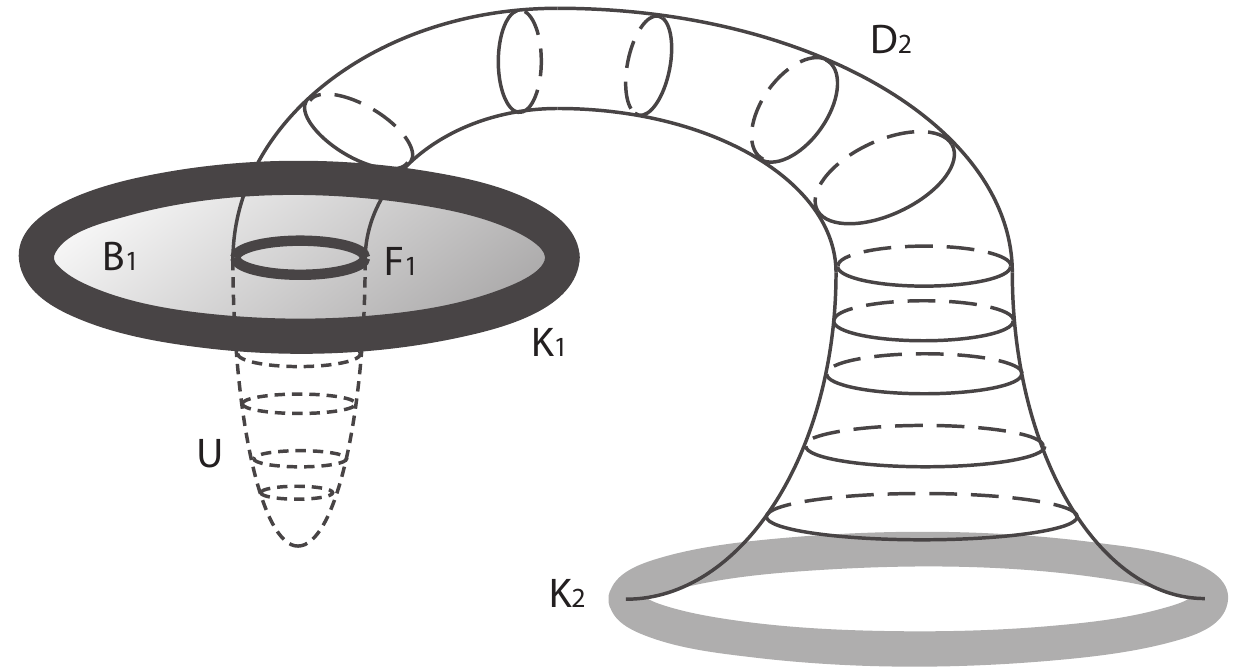}
	 \caption{We may cut out $U$ and paste in $f(U)$ to eliminate the intersections of $B_1$ and $D_2$.  
In this figure, set in $\R^3$, we have simplified the link by omitting $K_3$.}
	\label{fig:int1}
\end{figure}

\begin{figure}
	\centering
	\includegraphics[width=0.5\textwidth]{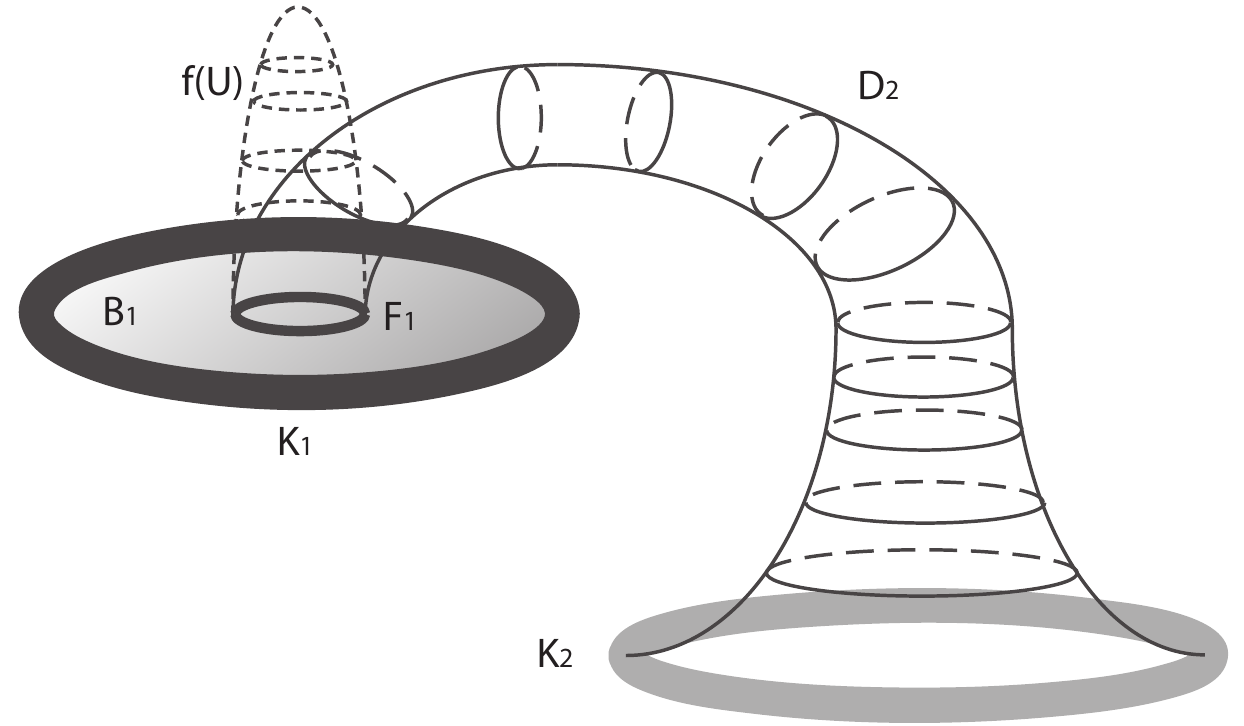}
	\caption{We reflect $U \subset D_2$ across $B_1$ to reduce the number of intersections.}
	\label{fig:int2}
\end{figure}

\begin{proof}[Proof of Theorem~\ref{thm:cbl}] We now complete the proof of our main theorem.  Observe that if $L_{i,n}$ is an unlink, then $K_2$ bounds a ball $D_2$ disjoint from $K_1 \cup K_3$, but Lemma~\ref{lemma:disjt} shows that this implies $K_2$ also bounds an immersed ball $D_2'$ that does not intersect $K_1 \cup K_3 \cup B_1$.  The existence of $D_2'$ implies that $K_2$ is homotopically trivial in the complement of $K_1 \cup K_3 \cup B_1$.  Since $K_3' \subset K_3 \cup B_1$, we know $K_2$ must also be homotopically trivial in the complement of $K_3'$.   This, however, contradicts Lemma~\ref{lemma:hopf}. 
\end{proof}

\smallskip
We note that by replacing our coefficients $(1, 4, 9,$ and $16)$ in the $L_{i,n}$ formulas by coefficients that are arbitrarily close to each one, say $(1, 1+ \epsilon, 1+ 2\epsilon,$ and $1+ 3\epsilon)$, we obtain embeddings of the same links containing two round spheres and a third that is arbitrarily close to being a round sphere.  Thus, although \cite{MR2384264} 
shows that no Brunnian link can be built out of round components, this paper provides an infinite collection of Brunnian links in which two of the components are round and the third is arbitrarily close to being round.

\section{A special case} \label{sec:proof2}

We have now proven our theorem, but a different proof technique exists for a subset of the links.  For the links $L_{1,n}$, one of the components of the link is homeomorphic to a circle, and we can utilize the fundamental group instead for our proof.  Here we allow $K_3$ to be the (elliptical) circle, but symmetry dictates that the proof holds if any of the other components had been the circle.  

\begin{theorem}
$L_{1,n}$ is a convex Brunnian link.
\label{thm:davis}
\end{theorem}

\begin{proof}
We want to show that the loop $K_3$ does not bound a disk disjoint from $K_1 \cup K_2$.  We do so by showing it is nontrivial in the fundamental group of  $\R^n - (K_1 \cup K_2)$.  We first prove the following lemma.

\begin{lemma}  
$\pi_1( \R^n - (K_1 \cup K_2))$ is the free group on two generators.
\label{lem:f2}
\end{lemma}

This follows from the following well known lemma.

\begin{lemma}
\label{lemma-svk}
The groups $\pi_1( S^n - J_i)$, $\pi_1( \R^n - K_i)$,  and $\pi_1( B^n - K_i)$ are all isomorphic to $\Z$ if $K_i$ is a round $(n-2)$-sphere in $B_n \subset \R^n$ and $J_i$ is a round $(n-2)$-sphere in $S^n$. 
\end{lemma}

\begin{proof}[Proof of Lemma~\ref{lemma-svk}]
Both the complement of $B^n$ and the boundary of $B^n$ are simply connected in $\R^n$ (and the same is true of a round ball in 
$S^n$) and thus the Seifert-Van Kampen Theorem shows that the fundamental group of the larger space is only dependent on the fundamental group of the ball.
Now we know that $\pi_1( S^n - J_i)$, $\pi_1( \R^n - K_i)$,  and $\pi_1( B^n - K_i)$ are all isomorphic to each other and we
need only show that $\pi_1( S^n - J_i)$ is isomorphic to $\Z$.  This follows from Lemma~\ref{lemma:dr}, which states that 
if we take a round sphere $S^{n-2} \subset S^n$, then $S^n - S^{n-2}$ deformation retracts onto $S^1$ and therefore has the fundamental group of a circle.
\end{proof}

\begin{proof}[Proof of Lemma~\ref{lem:f2}]
Since $K_1$ and $K_2$ are unlinked, there exists an embedded $S^{n-1}$ that separates them.  
Therefore the Seifert-Van Kampen Theorem implies that the fundamental group of $\R^n - (K_1 \cup K_2)$ is the free product of the fundamental groups of $B^n - K_1$ and $B^n - K_2$, i.e., the free group on two generators.
\end{proof}

Now to conclude the proof of Theorem~\ref{thm:davis}, we observe that in the case of $L_{1,n}$, the fundamental group of $\R^n - (K_1 \cup K_2)$ has generators
$$\begin{array}{ccl}
\alpha (t) &=& (\quad 3 \sin {2 \pi t}, \quad 0, \, 0 , \dots 0, \, 3 - 3 \cos{ 2 \pi t} ) \\
\beta (t)  & =& (2-2 \cos{2 \pi t}, \,  0, \, 0 , \dots 0,  \quad 2 \sin{ 2 \pi  t \quad)} 
\end{array} \quad \quad (0 \leq t \leq 1),$$
with the origin as the base point.  Let $\gamma (t)= (t, 0,0, \dots, 0)$.  We orient $K_3$ starting from the point $(1,0, \dots, 0)$ and move initially in the positive $x_n$ direction.

We observe that $\gamma K_3 \gamma^{-1}$ is homotopic to $\alpha \beta^{-1} \alpha ^{-1} \beta$,
which is a commutator of the two generators.  
(The circle $\alpha$ is clearly homotopic to the loop made by following the top half of $K_3$ then closing it by following the $x_1$-axis.  
The hard part of this observation is noticing that $\beta^{-1} \alpha ^{-1} \beta$ is homotopic to a curve that follows the bottom half of ellipse $K_3$ and then returns along the $x_1$-axis.)
Since the commutator is nontrivial, $K_3$ must represent a nontrivial element in $\pi_1 \left(\R^n - (K_1 \cup K_2) \right)$.  Thus, $L_{1,n}$ is not an unlink but its proper sublinks are unlinks; hence is $L_{1,n}$ is Brunnian.
\end{proof}

\section{Open Questions} \label{sec:open}

Here we list a few conjectures regarding the following infinite family of convex links.  
\begin{equation} 
\label{thirdlink}
L_{i,j,n}= K_1 \cup K_2 \cup K_3 
\end{equation}
\begin{align*}
K_1 & =  \{(x_1, x_2, \dots x_n): x_1^2 + x_2^2 + \dots + x_{n-j-1}^2 = 4, \; \, x_{n-j}  = \cdots = x_n=0 \} \\
K_2 &= \{(x_1, x_2, \dots x_n): x_{i+1}^2 + x_{i+2}^2 + \dots + x_{n}^2 = 9, \;\; x_1=x_2 = \cdots = x_{i}=0 \} \\
K_3 &= \{(x_1, x_2, \dots x_n): x_1^2 + x_2^2 +\cdots + x_{i}^2 + \frac{x_{n-j}^2 + \cdots + x_{n-1}^2  + x_n^2}{16} = 1, \\
 & \qquad \quad x_{i+1} =x_{i+2}= \cdots =x_{n-j-1}=0 \}
\end{align*}

These links generalize the family $L_{i,n}$ considered in section~\ref{sec:main}, in that $L_{i,0,n}=L_{i,n}$.

\begin{conjecture} 
All the links in the family $L_{i,j,n}$ are Brunnian.  
\end{conjecture}

\begin{conjecture} 
There exist convex Brunnian links that are not isotopic to a link of the form $L_{i,j,n}$.  
\end{conjecture}

No such links exist in $\R^3$ for 3, 4, and 5 component links by \cite{MR2287436, davis}.  We speculate that adding more components in $\R^3$ will not produce examples.  However, in higher dimensions it seems likely that other convex Brunnian links exist.  In particular,  is there such a link comprised of three $(n-2)$-dimensional knots for $n>3$?  Aside from the Borromean rings in $\R^3$, this case is impossible within our families \eqref{firstlink} and \eqref{thirdlink}.

\begin{conjecture} 
Although no Brunnian link can be built out of round spheres (see \cite{MR2384264}), it is true that for any $\epsilon >0 $, our families \eqref{firstlink} and \eqref{thirdlink} of links may be isotoped so that $K_1$ and $K_2$ remain as round spheres and $K_3$ is contained in an $\epsilon$-neighborhood of a round sphere.  We conjecture that all convex Brunnian links can be made arbitrarily close to perfectly round in this manner.
\end{conjecture}

\footnotesize
\bibliography{convexbrunnian}
\bibliographystyle{plain}

\end{document}